\documentclass{amsart}
\usepackage{amsmath}
\usepackage{amsthm}
\usepackage{amssymb}

\usepackage{tikz}
\usetikzlibrary{arrows}

\usepackage{clontzDefinitions}

\renewcommand{\vec}{\mathbf}

      \theoremstyle{plain}
      \newtheorem{theorem}{Theorem}
      \newtheorem{lemma}[theorem]{Lemma}
      \newtheorem{corollary}[theorem]{Corollary}
      \newtheorem{proposition}[theorem]{Proposition}
      
      \newtheorem{question}[theorem]{Question}

      \theoremstyle{definition}
      \newtheorem{definition}[theorem]{Definition}
      \newtheorem{notation}[theorem]{Notation}

      \theoremstyle{remark}

      \theoremstyle{plain}
      \newtheorem*{theorem*}{Theorem}
      \newtheorem*{lemma*}{Lemma}
      \newtheorem*{corollary*}{Corollary}
      \newtheorem*{proposition*}{Proposition}
      \newtheorem*{conjecture*}{Conjecture}
      \newtheorem*{question*}{Question}
      \newtheorem*{claim*}{Claim}

      \theoremstyle{definition}
      \newtheorem*{definition*}{Definition}
      \newtheorem*{example*}{Example}
      \newtheorem*{game*}{Game}

      \theoremstyle{remark}
      \newtheorem*{remark*}{Remark}

\begin{document}

\title{Relating games of Menger, countable fan tightness, and selective separability}



\author{Steven Clontz}
\address{Department of Mathematics and Statistics,
The University of South Alabama,
Mobile, AL 36606}
\email{steven.clontz@gmail.com}


\begin{abstract}
  By adapting techniques of Arhangel'skii, Barman, and Dow, we may
  equate the existence of perfect-information, Markov, and tactical
  strategies between two interesting selection games.
  These results shed some light on Gruenhage's question asking whether all
  strategically selectively separable spaces are Markov selectively
  separable.
\end{abstract}

\maketitle

\section{Introduction}

\begin{definition}
  The \term{selection principle} \(\schSelProp{\mc A}{\mc B}\) states that
  given \(A_n\in\mc A\) for \(n<\omega\), there exist \(B_n\in[A_n]^{<\omega}\)
  such that \(\bigcup_{n<\omega}B_n\in\mc B\).
\end{definition}

\begin{definition}
  The \term{selection game} \(\schSelGame{\mc A}{\mc B}\) is the
  analogous game to \(\schSelProp{\mc A}{\mc B}\), where during each
  round \(n<\omega\), Player \(\plI\) first
  chooses \(A_n\in\mc A\), and then Player \(\plII\) chooses
  \(B_n\in[A_n]^{<\omega}\).
  Player \(\plII\) wins in the case that \(\bigcup_{n<\omega}B_n\in\mc B\),
  and Player \(\plI\) wins otherwise.
\end{definition}

This game and property were first formally investigated
by Scheepers in ``Combinatorics of open covers'' \cite{MR1378387}, which
inspired a series of ten sequels with several co-authors. We may quickly
observe that if \(\plII\) has a winning strategy for the game
\(\schSelGame{\mc A}{\mc B}\), then \(\schSelProp{\mc A}{\mc B}\) will hold,
but the converse need not follow.

The power of this selection principle and game comes from their ability
to characterize several properties and games from the literature.
Of interest to us are the following.

\begin{definition}
  Let \(\mc O_X\) be the collection of open covers for a topological space
  \(X\). Then \(\schSelProp{\mc O_X}{\mc O_X}\) is the well-known
  \term{Menger property} for \(X\) (\(M\) for short), and
  \(\schSelGame{\mc O_X}{\mc O_X}\) is the
  well-known \term{Menger game}.
\end{definition}

\begin{definition}
  An \term{\(\omega\)-cover} \(\mc U\)
  for a topological space \(X\) is an open cover
  such that for every \(F\in[X]^{<\omega}\), there exists some \(U\in\mc U\)
  such that \(F\subseteq U\).
\end{definition}

\begin{definition}
  Let \(\Omega_X\) be the collection of \(\omega\)-covers for a topological
  space \(X\). Then \(\schSelProp{\Omega_X}{\Omega_X}\) is the
  \term{\(\Omega\)-Menger property} for \(X\) (\(\Omega M\) for short), and
  \(\schSelGame{\Omega_X}{\Omega_X}\) is the \term{\(\Omega\)-Menger game}.
\end{definition}

In \cite[Theorem 3.9]{MR1419798} it was shown that \(X\) is \(\Omega\)-Menger
if and only if \(X^n\) is Menger for all \(n<\omega\).

\begin{definition}
  Let \(\mc B_{X,x}\) be the collection of subsets \(A\subset X\) where
  \(x\in\closure{A}\). (Call \(A\) a \term{blade} of \(x\).)
  Then \(\schSelProp{\mc B_{X,x}}{\mc B_{X,x}}\) is the
  \term{countable fan tightness property} for \(X\) at \(x\)
  (\(CFT_x\) for short), and
  \(\schSelGame{\mc B_{X,x}}{\mc B_{X,x}}\) is the
  \term{countable fan tightness game} for \(X\) at \(x\).
\end{definition}

\begin{definition}
  A space \(X\) has \term{countable fan tightness} (\(CFT\) for short)
  if it has
  countable fan tightness at each point \(x\in X\).
\end{definition}

\begin{definition}
  Let \(\mc D_{X}\) be the collection of dense subsets of a topological
  space \(X\). (So, \(\mc D_X\subseteq \mc B_{X,x}\) for all \(x\in X\).)
  Then \(\schSelProp{\mc D_X}{\mc B_{X,x}}\) is the
  \term{countable dense fan tightness property} for \(X\) at \(x\)
  (\(CDFT_x\) for short), and
  \(\schSelGame{\mc D_X}{\mc B_{X,x}}\) is the
  \term{countable dense fan tightness game} for \(X\) at \(x\).
\end{definition}

\begin{definition}
  A space \(X\) has \term{countable dense fan tightness}
  (\(CDFT\) for short) if it has
  countable dense fan tightness at each point \(x\in X\).
\end{definition}

Note that \(CFT\Rightarrow CDFT\) for any space \(X\) as
\(\mc D_X\subseteq \mc B_{X,x}\).

The notion of countable fan tightness was first studied by
by Arhangel'skii in \cite{MR837289}. A result of that paper showed
that for \(T_{3\frac{1}{2}}\) spaces \(X\), the countable fan tightness
of the space of real-vaued continuous functions with
pointwise convergence \(C_p(X)\) is characterized by
the \(\Omega\)-Menger property of \(X\).

\begin{definition}
  \(\schSelProp{\mc D_X}{\mc D_X}\) is the
  \term{selective separability property} for \(X\)
  (\(SS\) for short), and
  \(\schSelGame{\mc D_X}{\mc D_X}\) is the
  \term{selective separability game} for \(X\).
\end{definition}

Of course, one may easily observe that a selective separable space is
separable. In \cite{MR2678950} Barman and Dow demonstrated that all
separable Frechet spaces are selectively separable. They were also able
to produce a space which is selectively separable, but does not allow
\(\plII\) a winning strategy in the selective separability game.

The object of this paper is to investigate the game-theoretic properties
characterized by the presence of winning \term{limited information}
strategies in these selection games.

\begin{definition}
  A \term{strategy} for \(\plII\) in the game \(\schSelGame{\mc A}{\mc B}\)
  is a function \(\sigma\) satisfying
  \(\sigma(\<A_0,\dots,A_n\>)\in[A_n]^{<\omega}\) for
  \(\<A_0\,\dots,A_n\>\in\mc A^{n+1}\). We say this strategy is
  \term{winning} if whenever \(\plI\) plays \(A_n\in\mc A\) during each
  round \(n<\omega\), \(\plII\) wins the game by playing
  \(\sigma(\<A_0,\dots,A_n\>)\) during each round \(n<\omega\).
  If a winning strategy exists, then we write
  \(\plII\win\schSelGame{\mc A}{\mc B}\).
\end{definition}

\begin{definition}
  A \term{Markov strategy} for \(\plII\) in the game
  \(\schSelGame{\mc A}{\mc B}\)
  is a function \(\sigma\) satisfying
  \(\sigma(A,n)\in[A_n]^{<\omega}\) for
  \(A\in\mc A\) and \(n<\omega\). We say this Markov strategy is
  \term{winning} if whenever \(\plI\) plays \(A_n\in\mc A\) during each
  round \(n<\omega\), \(\plII\) wins the game by playing
  \(\sigma(A_n,n)\) during each round \(n<\omega\).
  If a winning Markov strategy exists, then we write
  \(\plII\markwin\schSelGame{\mc A}{\mc B}\).
\end{definition}

\begin{notation}
  If \(\schSelProp{\mc A}{\mc B}\) characterizes the property \(P\),
  then we say \(\plII\win\schSelGame{\mc A}{\mc B}\) characterizes
  \(P^+\) (``strategically \(P\)''), and
  \(\plII\markwin\schSelGame{\mc A}{\mc B}\) characterizes
  \(P^{+mark}\) (``Markov \(P\)'').
  Of course, \(P^{+mark}\Rightarrow P^+ \Rightarrow P\).
\end{notation}

In this notation,
Barman and Dow showed that \(SS\) does not imply
\(SS^+\). Our goal is to make progress on the following question
attributed to Gary Gruenhage:

\begin{question}\label{mainQuestion}
  Does \(SS^+\) imply \(SS^{+mark}\)?
\end{question}

The solution is already known to be ``yes'' in the context of countable spaces
\cite{MR2678950}. However in general, winning strategies in selection games
cannot be improved to be winning Markov strategies. In
\cite{clontzMengerGamePreprint} the author showed that while \(M^+\) implies
\(M^{+mark}\) for second-countable spaces, there exists a simple
example of a regular non-second-countable space which is
\(M^+\) but not \(M^{+mark}\).

\section{\(CFT\), \(CDFT\) and \(SS\)}

We begin by generalizing the following result:

\begin{theorem}[Lem 2.7 of \cite{MR2678950}]
  The following are equivalent for any topological space \(X\).
  \begin{itemize}
    \item \(X\) is \(SS\).
    \item \(X\) is separable and \(CDFT\).
    \item \(X\) has a countable dense subset \(D\) where
          \(CDFT_x\) holds for all \(x\in D\).
  \end{itemize}
\end{theorem}

\begin{theorem}
  The following are equivalent for any topological space \(X\).
  \begin{itemize}
    \item \(X\) is \(SS\) (resp. \(SS^+\), \(SS^{+mark}\)).
    \item \(X\) is separable and \(CDFT\)
          (resp. \(CDFT^+\), \(CDFT^{+mark}\)).
    \item \(X\) has a countable dense subset \(D\) where
          \(CDFT_x\) (resp. \(CDFT_x^+\), \(CDFT_x^{+mark}\))
          holds for all \(x\in D\).
  \end{itemize}
\end{theorem}

\begin{proof}
  We need only show that the final condition implies the first.
  Let \(D=\{d_i:i<\omega\}\).

  Let \(\sigma_i\) be a witness for \(CDFT_{d_i}^+\)
  for each \(i<\omega\). We define the strategy \(\tau\) for the
  \(SS\) game by
  \[
    \tau(\<D_0,\dots,D_n\>)
      =
    \bigcup_{i\leq n}
    \sigma_i(\<D_i,\dots,D_n\>)
  .\]

  Let \(\<D_0,D_1,\dots\>\in\mc D_X^\omega\).
  By \(CDFT_{d_i}^+\), we have
  \[
    d_i
      \in
    \cl{\bigcup_{i\leq n<\omega}\sigma_i(\<D_i,\dots,D_n\>)}
      \subseteq
    \cl{\bigcup_{i\leq n<\omega}\tau(\<D_0,\dots,D_n\>)}
      \subseteq
    \cl{\bigcup_{n<\omega}\tau(\<D_0,\dots,D_n\>)}
  \]
  and as \(D\subseteq\cl{\bigcup_{n<\omega}\tau(\<D_0,\dots,D_n\>)}\)
  it follows that
  \[
    X
      \subseteq
    \cl D
      \subseteq
    \cl{\cl{\bigcup_{n<\omega}\tau(\<D_0,\dots,D_n\>)}}
      =
    \cl{\bigcup_{n<\omega}\tau(\<D_0,\dots,D_n\>)}
  .\]
  Therefore \(\tau\) witnesses \(SS^+\).

  Now let \(\sigma_i\) be a witness for \(CDFT_{d_i}^{+mark}\)
  for each \(i<\omega\). We define the Markov strategy \(\tau\) for the
  \(SS\) game by
  \[
    \tau(D,n)
      =
    \bigcup_{i\leq n}
    \sigma_i(D,n-i)
  .\]

  Let \(\<D_0,D_1,\dots\>\in\mc D_X^\omega\).
  By \(CDFT_{d_i}^{+mark}\), we have
  \[
    d_i
      \in
    \cl{\bigcup_{i\leq n<\omega}\sigma_i(D_n,n-i)}
      \subseteq
    \cl{\bigcup_{i\leq n<\omega}\tau(D_n,n)}
      \subseteq
    \cl{\bigcup_{n<\omega}\tau(D_n,n)}
  \]
  and as \(D\subseteq\cl{\bigcup_{n<\omega}\tau(D_n,n)}\)
  it follows that
  \[
    X
      \subseteq
    \cl D
      \subseteq
    \cl{\cl{\bigcup_{n<\omega}\tau(D,n)}}
      =
    \cl{\bigcup_{n<\omega}\tau(D,n)}
  .\]
  Therefore \(\tau\) witnesses \(SS^{+mark}\).
\end{proof}

So amongst separable spaces, we see that \(SS\)
(resp. \(SS^+\), \(SS^{+mark}\)) and
\(CDFT\) (resp. \(CDFT^+\), \(CDFT^{+mark}\))
are equivalent. We now further bridge the gap between \(CDFT\)
and \(CFT\) in the context of function spaces.
Consider the following result of Arhangel'skii.

\begin{theorem}[\cite{MR837289}]
  The following are equivalent for any \(T_{3\frac{1}{2}}\)
  topological space \(X\).
    \begin{itemize}
      \item \(X\) is \(\Omega M\).
      \item \(C_p(X)\) is \(CFT\).
    \end{itemize}
\end{theorem}

This result may similarly be generalized in a game theoretic sense.
In addition, this proof will demonstrate the equivalence of
\(CFT\) and \(CDFT\) in \(C_p(X)\).
It is unknown to the author whether Arhangel'skii
used a strategy similar to the following proof in \cite{MR837289},
but Sakai employed a similar technique
in \cite{MR964873} to relate the \(\Omega\)-Rothberger and
countable strong fan tightness properties
(and essentially, the countable strong dense fan tightness property).
Due to the difficulty in obtaining an English translation of
\cite{MR837289}, we reprove Arhangel'skii's theorem above in our
more general context below.

\begin{proposition}
  Any witness for \(\Omega M\) (resp. \(\Omega M^+,\Omega M^{+mark}\))
  may be improved such that any final sequence of the chosen
  finite subcollections is an \(\omega\)-cover.
\end{proposition}

\begin{proof}
  Consider any sequence \(\<\mc U_0,\mc U_1,\dots\>\) of open covers.
  For \(\Omega M\), choose the witness
  \(\<\mc F_n^n,\mc F_{n+1}^n,\dots\>\) for the final
  sequence \(\<\mc U_n,\mc U_{n+1},\dots\>\) for each \(n<\omega\),
  and let \(\mc F_n=\bigcup_{i\leq n}\mc F_n^i\).
  For \(\Omega M^+\), the winning strategy \(\sigma\) may be improved to
  \(\tau\) where
  \(
    \tau(\<\mc U_0,\dots,\mc U_n\>)
      =
    \bigcup_{i\leq n}\sigma(\<\mc U_i,\dots,\mc U_n\>)
  \). For \(\Omega M^{+mark}\), the winning Markov strategy \(\sigma\) may
  be improved to \(\tau\) where
  \(
    \tau(\mc U,n)
      =
    \bigcup_{i\leq n}\sigma(\mc U,i)
  \).
\end{proof}

\begin{definition}
  Let \(X\) be a \(T_{3\frac{1}{2}}\) topological space.
  For \(\vec x\in C_p(X)\), \(F\in[X]^{<\omega}\), and
  \(\epsilon>0\), let
  \[
    [\vec x,F,\epsilon]
      =
    \{
      \vec y\in C_p(x)
    :
      |\vec y(t)-\vec x(t)|<\epsilon
    \text{ for all }
      t\in F
    \}
  \]
  give a basic open neighborhood of \(\vec x\).
\end{definition}

\begin{lemma}\label{mengerFanLemma1}
  Let \(X\) be a \(T_{3\frac{1}{2}}\) topological space.
  If \(X\) is \(\Omega M\)
  (resp. \(\Omega M^+\), \(\Omega M^{+mark}\)).
  then \(C_p(X)\) is \(CFT_{\vec 0}\)
  (resp. \(CFT_{\vec 0}^+\), \(CFT_{\vec 0}^{+mark}\)).
\end{lemma}

\begin{proof}
  For each \(B\in\mc B_{C_p(X),\vec 0}\) define
  \[
    \mc U_n(B)
      =
    \left\{
      \vec x^{\leftarrow}\left[\left(-\frac{1}{2^n},\frac{1}{2^n}\right)\right]
    :
      \vec x\in B
    \right\}
  .\]
  Consider the finite set \(F\in[X]^{<\omega}\). Since
  \(\vec 0\in\cl B\), choose \(\vec x\in B\cap[\vec 0,F,1]\).
  It follows that \(F\subseteq x^{\leftarrow}[(-1,1)]\in\mc U_n(B)\),
  so \(\mc U_n(B)\) is an \(\omega\)-cover of \(X\).

  Consider the sequence of blades
  \(\<B_0,B_1,\dots\>\in\mc B_{C_p(X),\vec 0}^\omega\), and the
  corresponding sequence of \(\omega\)-covers
  \(\<\mc U_0(B_0),\mc U_1(B_1),\dots\>\in \Omega_X^\omega\).

  Assuming \(X\) is \(\Omega M\), choose a witness
  \(\<\mc F_0,\mc F_1,\dots\>\) such that
  \(\mc F_n\in[\mc U_n(B_n)]^{<\omega}\) and
  \[
    \bigcup_{i\leq n<\omega} \mc F_n
  \]
  is an \(\omega\)-cover of \(X\) for all \(i<\omega\).
  Now let
  \[
    F_n
      =
    \left\{
      \vec x\in B_n
    :
      \vec x^{\leftarrow}\left[\left(-\frac{1}{2^n},\frac{1}{2^n}\right)\right]
        \in
      \mc F_n
    \right\}
  .\]
  We claim that \(\vec 0\in\cl{\bigcup_{n<\omega}F_n}\).
  Let \(G\in[X]^{<\omega},\epsilon>0\). Choose \(i<\omega\) such that
  \(\frac{1}{2^i}<\epsilon\) and then choose
  \(i\leq n<\omega,\vec x^{\leftarrow}[(-\frac{1}{2^n},\frac{1}{2^n})]\in\mc F_n\)
  such that \(G\subseteq \vec x^{\leftarrow}[(-\frac{1}{2^n},\frac{1}{2^n})]\)
  and therefore \(G\subseteq \vec x^{\leftarrow}[(-\epsilon,\epsilon)]\).
  It follows that \(\vec x\in F_n\cap[\vec 0,G,\epsilon]\), and therefore
  every basic open neighborhood of \(\vec 0\) intersects
  \(\bigcup_{n<\omega}F_n\).

  Assuming \(X\) is \(\Omega M^+\), choose a witness
  \(\sigma\) such that
  \[
    \bigcup_{i\leq n<\omega} \sigma(\<\mc U_0(B_0),\dots,\mc U_n(B_n)\>)
  \]
  is an \(\omega\)-cover of \(X\) for all \(i<\omega\).
  Now let
  \[
    \tau(\<B_0,\dots,B_n\>)
      =
    \left\{
      \vec x\in B_n
    :
      \vec x^{\leftarrow}\left[\left(-\frac{1}{2^n},\frac{1}{2^n}\right)\right]
        \in
      \sigma(\<\mc U_0(B_0),\dots,\mc U_n(B_n)\>)
    \right\}
  .\]
  We claim that \(\vec 0\in\cl{\bigcup_{n<\omega}\tau(\<B_0,\dots,B_n\>)}\).
  Let \(G\in[X]^{<\omega},\epsilon>0\). Choose \(i<\omega\) such that
  \(\frac{1}{2^i}<\epsilon\) and then choose
  \(i\leq n<\omega,\vec x^{\leftarrow}[(-\frac{1}{2^n},\frac{1}{2^n})]\in\sigma(\<\mc U_0(B_0),\dots,\mc U_n(B_n)\>)\)
  such that \(G\subseteq \vec x^{\leftarrow}[(-\frac{1}{2^n},\frac{1}{2^n})]\)
  and therefore \(G\subseteq \vec x^{\leftarrow}[(-\epsilon,\epsilon)]\).
  It follows that \(\vec x\in \tau(\<B_0,\dots,B_n\>)\cap[\vec 0,G,\epsilon]\), and therefore
  every basic open neighborhood of \(\vec 0\) intersects
  \(\bigcup_{n<\omega}\tau(\<B_0,\dots,B_n\>)\).

  Assuming \(X\) is \(\Omega M^{+mark}\), choose a witness
  \(\sigma\) such that
  \[
    \bigcup_{i\leq n<\omega} \sigma(\mc U_n(B_n),n)
  \]
  is an \(\omega\)-cover of \(X\) for all \(i<\omega\).
  Now let
  \[
    \tau(B_n,n)
      =
    \left\{
      \vec x\in B_n
    :
      \vec x^{\leftarrow}\left[\left(-\frac{1}{2^n},\frac{1}{2^n}\right)\right]
        \in
      \sigma(\mc U_n(B_n),n)
    \right\}
  .\]
  We claim that \(\vec 0\in\cl{\bigcup_{n<\omega}\tau(B_n,n)}\).
  Let \(G\in[X]^{<\omega},\epsilon>0\). Choose \(i<\omega\) such that
  \(\frac{1}{2^i}<\epsilon\) and then choose
  \(i\leq n<\omega,\vec x^{\leftarrow}[(-\frac{1}{2^n},\frac{1}{2^n})]\in\sigma(\mc U_n(B_n),n)\)
  such that \(G\subseteq \vec x^{\leftarrow}[(-\frac{1}{2^n},\frac{1}{2^n})]\)
  and therefore \(G\subseteq \vec x^{\leftarrow}[(-\epsilon,\epsilon)]\).
  It follows that \(\vec x\in \tau(B_n,n)\cap[\vec 0,G,\epsilon]\), and therefore
  every basic open neighborhood of \(\vec 0\) intersects
  \(\bigcup_{n<\omega}\tau(B_n,n)\).
\end{proof}

\begin{lemma}\label{mengerFanLemma2}
  Let \(X\) be a \(T_{3\frac{1}{2}}\) topological space.
  If \(C_p(X)\) is \(CDFT_{\vec 0}\)
  (resp. \(CDFT_{\vec 0}^+\), \(CDFT_{\vec 0}^{+mark}\)),
  then \(X\) is \(\Omega M\)
  (resp. \(\Omega M^+\), \(\Omega M^{+mark}\)).
\end{lemma}

\begin{proof}
  For each \(\mc U\in\Omega_{X}\) define
  \[
    D(\mc U)
      =
    \{
      \vec y\in C_p(X)
    :
      \vec y[X\setminus U_{\vec y,\mc U}]=\{1\}
      \text{ for some }
      U_{\vec y,\mc U}\in\mc U
    \}
  .\]
  Consider the point \(\vec x\in C_p(X)\) and its basic open neighborhood
  \([\vec x,G,\epsilon]\). If \(\mc U\) is an \(\omega\)-cover
  of \(X\), \(G\subseteq U\) for some \(U_{\vec y,\mc U}\in\mc U\).
  Since \(X\) is
  \(T_{3\frac{1}{2}}\), \(X\setminus U_{\vec y,\mc U}\) is closed, and \(G\)
  is finite and disjoint from \(X\setminus U_{\vec y,\mc U}\),
  we may choose some function \(\vec y\in C_p(X)\) where
  \(\vec y[X\setminus U_{\vec y,\mc U}]=\{1\}\) and \(\vec x(t)=\vec y(t)\)
  for each \(t\in G\).
  It follows \(\vec y\in [\vec x,G,\epsilon]\cap D\), so \(D(\mc U)\)
  is dense in \(C_p(X)\).

  Consider the sequence of \(\omega\)-covers
  \(\<\mc U_0,\mc U_1,\dots\>\in\Omega_X^\omega\), and the
  corresponding sequence of dense subsets
  \(\<D(\mc U_0),D(\mc U_1),\dots\>\in\mc D_{C_p(X)}^\omega\).

  Assuming \(C_p(X)\) is \(CDFT_{\vec 0}\), choose a witness
  \(\<F_0,F_1,\dots\>\) such that
  \[
    \vec 0 \in \cl{\bigcup_{n<\omega} F_n}
  .\]
  Now let
  \[
    \mc F_n
      =
    \{
      U_{\vec y,\mc U_n}
    :
      \vec y \in F_n
    \}
      \in
    [\mc U_n]^{<\omega}
  .\]
  We claim that \(\bigcup_{n<\omega}\mc F_n\) is an \(\omega\)-cover.
  Let \(G\in[X]^{<\omega}\). The neighborhood \([\vec 0,G,\frac{1}{2}]\)
  contains some point \(\vec y\in F_n\) for some \(n<\omega\). It follows
  that \(U_{\vec y,\mc U_n}\in\mc U_n\) and
  \(\vec y[X\setminus U_{\vec y,\mc U_n}]=\{1\}\). It follows that
  \(G\cap(X\setminus U_{\vec y,\mc U_n})=\emptyset\), and therefore
  \(G\subseteq U_{\vec y,\mc U_n}\in\mc F_n\).

  Assuming \(C_p(X)\) is \(CDFT_{\vec 0}^+\), choose a witness
  \(\sigma\) such that
  \[
    \vec 0
      \in
    \cl{\bigcup_{n<\omega} \sigma(\<D(\mc U_0),\dots,D(\mc U_n)\>)}
  .\]
  Now let
  \[
    \tau(\<\mc U_0,\dots,\mc U_n\>)
      =
    \{
      U_{\vec y,\mc U_n}
    :
      \vec y \in \sigma(\<D(\mc U_0),\dots,D(\mc U_n)\>)
    \}
      \in
    [\mc U_n]^{<\omega}
  .\]
  We claim that \(\bigcup_{n<\omega}\tau(\<\mc U_0,\dots,\mc U_n\>)\)
  is an \(\omega\)-cover.
  Let \(G\in[X]^{<\omega}\). The neighborhood \([\vec 0,G,\frac{1}{2}]\)
  contains some point \(\vec y\in \sigma(\<D(\mc U_0),\dots,D(\mc U_n)\>)\)
  for some \(n<\omega\). It follows
  that \(U_{\vec y,\mc U_n}\in\mc U_n\) and
  \(\vec y[X\setminus U_{\vec y,\mc U_n}]=\{1\}\). As a result
  \(G\cap(X\setminus U_{\vec y,\mc U_n})=\emptyset\), and therefore
  \(G\subseteq U_{\vec y,\mc U_n}\in\mc \tau(\<\mc U_0,\dots,\mc U_n\>)\).

  Assuming \(C_p(X)\) is \(CDFT_{\vec 0}^{+mark}\), choose a witness
  \(\sigma\) such that
  \[
    \vec 0
      \in
    \cl{\bigcup_{n<\omega} \sigma(D(\mc U_n),n)}
  .\]
  Now let
  \[
    \tau(\mc U_n,n)
      =
    \{
      U_{\vec y,\mc U_n}
    :
      \vec y \in \sigma(D(\mc U_n),n)
    \}
      \in
    [\mc U_n]^{<\omega}S
  .\]
  We claim that \(\bigcup_{n<\omega}\tau(\mc U_n,n)\)
  is an \(\omega\)-cover.
  Let \(G\in[X]^{<\omega}\). The neighborhood \([\vec 0,G,\frac{1}{2}]\)
  contains some point \(\vec y\in \sigma(D(\mc U_n),n)\)
  for some \(n<\omega\). It follows
  that \(U_{\vec y,\mc U_n}\in\mc U_n\) and
  \(\vec y[X\setminus U_{\vec y,\mc U_n}]=\{1\}\). As a result
  \(G\cap(X\setminus U_{\vec y,\mc U_n})=\emptyset\), and therefore
  \(G\subseteq U_{\vec y,\mc U_n}\in\mc \tau(\mc U_n,n)\).
\end{proof}

\begin{theorem}\label{mengerFanTheorem}
  The following are equivalent for any \(T_{3\frac{1}{2}}\)
  topological space \(X\).
    \begin{itemize}
      \item \(X\) is \(\Omega M\)
            (resp. \(\Omega M^+\), \(\Omega M^{+mark}\)).
      \item \(C_p(X)\) is \(CFT\)
            (resp. \(CFT^+\), \(CFT^{+mark}\)).
      \item \(C_p(X)\) is \(CDFT\)
            (resp. \(CDFT^+\), \(CDFT^{+mark}\)).
    \end{itemize}
\end{theorem}

\begin{proof}
  Since \(\mc D_X\subseteq \mc B_{X,x}\), the second condition trivially
  implies the first. As \(C_p(X)\) is homogeneous, the \(C(D)FT\) properties
  follow from \(C(D)FT_{\vec 0}\). So the result follows from the
  previous lemmas.
\end{proof}

\section{A space which is \(CFT^+\) but not \(CDFT^{+mark}\)}

Our goal is to now demonstrate a space which is \(CFT^+\),
but not even \(CDFT^{+mark}\). To do so, we will find a space \(X\)
which is \(\Omega M^+\) but not \(\Omega M^{+mark}\), yielding \(C_p(X)\)
as our example.

\begin{proposition}\label{altCompactCharacterization}
  A space \(X\) is compact if and only if for each \(\omega\)-cover
  \(\mc U\) of \(X\) and \(n<\omega\), there exists a finite subcollection
  \(\mc F\in[\mc U]^{<\omega}\) such that for each \(F\in[X]^{\leq n}\)
  there exists some \(U\in\mc F\) where \(F\subseteq U\).
\end{proposition}

\begin{proof}
  Let \(X\) and therefore \(X^n\) be compact.
  Let \(\mc F_n\) be the finite subcover of the
  open cover \(\mc U_n=\{U^n:U\in\mc U\}\). Then
  \(\mc F=\{U:U^n\in\mc F_n\}\) witnesses our desired result.
\end{proof}

\begin{lemma}
  The following are equivalent for a regular topological space \(X\):
  \begin{itemize}
    \item \(X\) is \(\sigma\)-compact
    \item \(X\) is \(\Omega M^{+mark}\)
    \item \(X\) is \(M^{+mark}\)
  \end{itemize}
\end{lemma}

\begin{proof}
  The equivalence of \(\sigma\)-compact and \(M^{+mark}\) in regular spaces
  was shown in \cite{clontzMengerGamePreprint}. As \(\Omega M^{+mark}\)
  trivially implies \(M^{+mark}\), we need only demonstrate that
  if \(X=\bigcup_{n<\omega} K_n\) for \(K_n\) compact and increasing,
  then \(X\) is \(\Omega M^{+mark}\).

  We define \(\sigma(\mc U,n)\) for each \(\omega\)-cover \(\mc U\) and
  \(n<\omega\) to witness Proposition \ref{altCompactCharacterization}
  for \(\mc U\), \(K_n\), and \(n\). It follows that for every sequence of
  \(\omega\)-covers \(\<\mc U_0,\mc U_1,\dots\>\) and \(F\in[X]^{<\omega}\),
  \(F\in[K_n]^{\leq n}\) for some \(n\geq|F|\), and thus there exists some
  \(U\in\sigma(\mc U_n,n)\) where \(F\subseteq U\). Therefore
  \(\bigcup_{n<\omega}\sigma(\mc U,n)\) is an \(\omega\)-cover of \(X\).
\end{proof}

The reader may note that with this lemma, we may view Theorem
\ref{mengerFanTheorem} as a generalization of
\cite[Proposition 2.6]{MR2868880}.

\begin{definition}
  Let \(X\) be a topological space such that all countable sets are
  closed. Then \(\onePtLind{X}=X\cup\{\infty\}\)
  is the \term{one-point Lindel\"of-ication} of \(X\), with a basis given
  by the usual open sets of \(X\) and the co-countable sets containing
  \(\infty\).
\end{definition}

\begin{theorem}\label{omega1OmegaMenger}
  Let \(\omega_1\) have the discrete topology. Then
  \(\onePtLind{\omega_1}\) is \(\Omega M^+\) but not \(\Omega M^{+mark}\).
\end{theorem}

\begin{proof}
  The proof that \(X=\onePtLind{\omega_1}\) is not
  \(\sigma\)-compact (and therefore not \(\Omega M^{+mark}\))
  is simply the fact that its countably infinite subsets
  are closed and discrete.
  We define the strategy \(\sigma\) for \(\plII\) in
  \(\schSelGame{\Omega_X}{\Omega_X}\) as follows.

  For \(n<\omega\)
  let \(\sigma(\<\mc U_0,\dots,\mc U_n\>)=\{U(\<\mc U_0,\dots,\mc U_n\>)\}\),
  where
  \[
    U(\<\mc U_0,\dots,\mc U_n\>)
      =
    \onePtLind{\omega_1}
      \setminus
    \{\alpha_{n,m}:m<\omega\}
  \]
  is a co-countable set
  containing \(\{\infty\}\cup\{\alpha_{i,j}:i,j<n\}\).

  Consider the arbitrary sequence of moves \(\<\mc U_0,\mc U_1,\dots\>\)
  by \(\plI\). For \(F\in[\onePtLind{\omega_1}]^{<\omega}\), choose
  \(n<\omega\) such that
  \[
    F\cap\{\alpha_{i,j}:i,j<\omega\}
      =
    F\cap\{\alpha_{i,j}:i,j<n\}
      \subseteq
    U(\<\mc U_0,\dots,\mc U_n\>)
  \]
  It follows that as
  \[
    F
      \setminus
    \{\alpha_{i,j}:i,j<\omega\}
      \subseteq
    F
      \setminus
    \{\alpha_{n,m}:m<\omega\}
      \subseteq
    U(\<\mc U_0,\dots,\mc U_n\>)
  \]
  \(F\) is a subset of
  \(U(\<\mc U_0,\dots,\mc U_n\>)\), making
  \(
    \bigcup_{n<\omega}
    \sigma(\<\mc U_0,\dots,\mc U_n\>)
  \)
  an \(\omega\)-cover.
\end{proof}

\begin{corollary}
  \(C_p(\onePtLind{\omega_1})\) is \(CFT^+\) but not \(CDFT^{+mark}\).
\end{corollary}

If \(C_p(\onePtLind{\omega_1})\) were separable, it would be a negative
solution to Question \ref{mainQuestion}. However, it is not.

\begin{lemma}[\cite{MR953314}]
  For a \(T_{3\frac{1}{2}}\) topological space \(X\),
  \(C_p(X)\) is separable if and only if \(X\) has a coarser separable
  metrizable topology.
\end{lemma}

\begin{corollary}
  \(C_p(\onePtLind{\omega_1})\) is not separable.
\end{corollary}

\begin{proof}
  Every metrizable topological space has points \(G_\delta\).
  However, if every neighborhood of \(\infty\) is
  co-countable, \(\{\infty\}\) cannot be the intersection of countably many
  open sets.
\end{proof}

An affirmative answer to either of these questions would answer
Question \ref{mainQuestion} negatively.

\begin{question}
  Does there exist a separable subspace of \(C_p(\onePtLind{\omega_1})\)
  which is not \(CDFT^{+mark}\)?
\end{question}

\begin{question}
  Does there exist a non-\(\sigma\)-compact \(\Omega M^+\) space
  with a coarser separable metrizable topology?
\end{question}

\section{Equivalence of certain strategic and Markov seleciton properties}

Barman and Dow
previously demonstrated that \(SS^+\) is equivalent to
\(SS^{+mark}\) amongst countable spaces.
A similar result by the author
showed that \(M^+\) is equivalent to \(M^{+mark}\) in second-countable
spaces. The following result generalizes both.

\begin{lemma}\label{mainLemma}
  Let \(\mc A\) be a family of sets where
  \(|\bigcup\mc A|\leq\omega\)
  and let \(\mc B\) be a family of sets closed under supersets. Then
  \(\plII\win\schSelGame{\mc A}{\mc B}\) if and only if
  \(\plII\markwin\schSelGame{\mc A}{\mc B}\).
\end{lemma}

\begin{proof}
  Let \(\sigma\) witness \(\plII\win\schSelGame{\mc A}{\mc B}\).

  For \(t\in\omega^{<\omega}\), suppose
  \(Z_s\in\mc A\) has been defined for all \(s\leq t\). Note then that
  \(
    \{\sigma(\<Z_{t\rest1},Z_{t\rest2},\dots,Z_t,A\>:A\in\mc A\}
      \subseteq
    [\bigcup\mc A]^{<\omega}
  \)
  and therefore is countable.
  So choose \(Z_{t\concat\<n\>}\in\mc A\) for \(n<\omega\) such that
  \(
    \{\sigma(\<Z_{t\rest1},Z_{t\rest2},\dots,Z_t,A\>:A\in\mc A\}
      =
    \{\sigma(\<Z_{t\rest1},Z_{t\rest2},\dots,Z_t,Z_{t\concat\<n\>}\>):n<\omega\}
  \).

  Let \(b:\omega\to\omega^{<\omega}\) be a bijection,
  and define
  \(\tau(A,n)=\sigma(\<Z_{b(n)\rest 1},Z_{b(n)\rest 2},\dots,Z_{b(n)},A\>)\).
  Consider \(\<A_0,A_1,\dots\>\in\mc A^\omega\).

  Define \(f\in\omega^\omega\) as follows.
  Let \(n<\omega\) and suppose that \(f\rest n\) has been already
  been defined. Then choose \(f(n)<\omega\) such that
  \(
    \sigma(\<Z_{f\rest1},Z_{f\rest2},\dots,Z_{f\rest n},A_{b^{\leftarrow}(f\rest n)}\>)
      =
    \sigma(\<Z_{f\rest1},Z_{f\rest2},\dots,Z_{f\rest n},Z_{f\rest (n+1)}\>)
  \).

  Since \(\<Z_{f\rest1},Z_{f\rest2},\dots\>\in\mc A^\omega\), it follows
  that
  \(
    \bigcup_{n<\omega}\sigma(\<Z_{f\rest1},Z_{f\rest2},\dots,Z_{f\rest n}\>)
      \in
    \mc B
  \).
  The result then follows from
  \begin{align*}
    \bigcup_{n<\omega}\tau(A_n,n)
      & =
    \bigcup_{n<\omega}
    \sigma(\<Z_{b(n)\rest 1},Z_{b(n)\rest 2},\dots,Z_{b(n)},A_n\>)
      \\ & \supseteq
    \bigcup_{n<\omega,b(n)<f}
    \sigma(\<Z_{b(n)\rest 1},Z_{b(n)\rest 2},\dots,Z_{b(n)},A_n\>)
      \\ & =
    \bigcup_{n<\omega,b(n)<f}
    \sigma(\<Z_{f\rest 1},Z_{f\rest 2},\dots,Z_{f\rest|b(n)|},A_{b^{\leftarrow}(f\rest|b(n)|)}\>)
      \\ & =
    \bigcup_{n<\omega,b(n)<f}
    \sigma(\<Z_{f\rest 1},Z_{f\rest 2},\dots,Z_{f\rest|b(n)|},Z_{f\rest(|b(n)|+1)}\>)
      \\ & =
    \bigcup_{n<\omega}
    \sigma(\<Z_{f\rest 1},Z_{f\rest 2},\dots,Z_{f\rest n}\>)
      \\ & \in
    \mc B
  \end{align*}
  as \(\mc B\) is closed under supersets.
\end{proof}

\begin{corollary}[\cite{MR2868880}]
  A countable space is \(SS^+\) if and only if it is
  \(SS^{+mark}\).
\end{corollary}

\begin{proof}
  \(\bigcup\mc D_X=X\) is countable, and any set containing a dense
  set is dense.
\end{proof}

\begin{corollary}
  A second-countable space is \(\Omega M^+\) if and only if
  it is \(\Omega M^{+mark}\).
\end{corollary}

\begin{proof}
  First note that any witness for
  \(\Omega M^+\) is a witness for
  \(\plII\win\schSelGame{\Omega_X^*}{\Omega_X^*}\).
  Fix a countable base \(\mc B\) for \(X\) closed under finite unions, and let
  \(\Omega_X^*\) be the collection of \(\omega\)-covers of \(X\) which only
  use basic open sets. Then
  \(\bigcup\Omega_X^*=\mc B\) is countable, and any set containing an
  \(\omega\)-cover is an \(\omega\)-cover; therefore
  \(\plII\win\schSelGame{\Omega_X^*}{\Omega_X^*}\) if and only if
  \(\plII\markwin\schSelGame{\Omega_X^*}{\Omega_X^*}\).

  So let \(\sigma\) witness
  \(\plII\markwin\schSelGame{\Omega_X^*}{\Omega_X^*}\).
  For each \(\omega\)-cover \(\mc U\) and finite set \(F\in[X]^{<\omega}\),
  choose \(U_{\mc U,F}\in\mc U\) such that \(F\subseteq U\). Then choose a
  basic open set \(U_{\mc U,F}^*\) such that
  \(F\subseteq U_{\mc U,F}^*\subseteq U_{\mc U,F}\). Finally, for each
  \(\omega\)-cover \(\mc U\), let
  \[
    \mc U^*
      =
    \{
      U_{\mc U,F}^*
        :
      F\in[X]^{<\omega}
    \}
      \in
    \Omega_X^*
  .\]

  Define the strategy \(\tau\) such that
  \(\tau(\mc U,n)=\{U_{\mc U,F}:U_{\mc U,F}^*\in\sigma(\mc U^*,n)\}\).
  For any sequence of \(\omega\)-covers \(\<\mc U_0,\mc U_1,\dots\>\),
  it follows that \(\bigcup_{n<\omega}\sigma(\mc U_n^*,n)\) is an
  \(\omega\)-cover, and therefore
  \(\bigcup_{n<\omega}\tau(\mc U_n,n)\) is an \(\omega\)-cover also.
\end{proof}

\begin{corollary}[\cite{clontzMengerGamePreprint}]
  A second-countable space is \(M^+\) if and only if
  it is \(M^{+mark}\).
\end{corollary}

\begin{proof}
  An easy adapatation of the preceding proof, replacing \(\omega\)-covers
  with open covers, and replacing \(F\in[X]^{<\omega}\) with \(x\in X\).
\end{proof}

\section{Strong variants of selection princples and games}

Let \(\schStrongSelProp{\mc A}{\mc B},\schStrongSelGame{\mc A}{\mc B}\)
be the natural variants of
\(\schSelProp{\mc A}{\mc B},\schSelGame{\mc A}{\mc B}\) where each choice
by \(\plII\) must either be a single element or singleton
(whichever is more convenient for the proof at hand), rather than a finite
set. Convention calls for denoting these as
\term{strong} versions of the corresponding selection princples and games,
although the ``strong Menger'' property is commonly known as ``Rothberger''
who first investigated the principle in \cite{Rothberger1938}. We will thus
call ``strong \(\Omega\)-Menger'' ``\(\Omega\)-Rothberger'' and shorten it
with \(\Omega R\), and otherwise attach the prefix ``s''
when abbreviating to all other strong variants.

\begin{theorem}
  The following are equivalent for any topological space \(X\).
  \begin{itemize}
    \item \(X\) is \(sSS\) (resp. \(sSS^+\), \(sSS^{+mark}\)).
    \item \(X\) is separable and \(sCDFT\)
          (resp. \(sCDFT^+\), \(sCDFT^{+mark}\)).
    \item \(X\) has a countable dense subset \(D\) where
          \(sCDFT_x\) (resp. \(sCDFT_x^+\), \(sCDFT_x^{+mark}\))
          holds for all \(x\in D\).
  \end{itemize}
\end{theorem}

\begin{proof}
  We need only show that the final condition implies the first.
  Let \(D=\{d_i:i<\omega\}\).

  Let \(\{D_{m,n}\in\mc D_X:m,n<\omega\}\), and let \(x_{i,n}\in D_{i,n}\)
  witness \(sCDFT_{d_i}\), so
  \[
    d_i
      \in
    \cl{\{x_{i,n}:n<\omega\}}
      \subseteq
    \cl{\{x_{m,n}:m,n<\omega\}}
  \]
  and as \(D\subseteq\cl{\{x_{m,n}:m,n<\omega\}}\)
  it follows that
  \[
    X
      \subseteq
    \cl D
      \subseteq
    \cl{\cl{\{x_{m,n}:m,n<\omega\}}}
      =
    \cl{\{x_{m,n}:m,n<\omega\}}
  .\]
  Therefore \(x_{m,n}\in D_{m,n}\) witnesses \(sSS\).

  Let \(\sigma_i\) be a witness for \(sCDFT_{d_i}^+\)
  for each \(i<\omega\). Define \(p:\omega\to\omega\) partition \(\omega\)
  into infinite sets \(\{p(i):i<\omega\}\).
  For a nonempty finite sequence \(t\), let \(t'\) be its subsequence removing
  all terms of index \(n\) such that \(p(n)\not=p(|t|-1)\).
  (Note that this process preserves the final term.)
  We then define the strategy \(\tau\) by
  \[
    \tau(t)
      =
    \sigma_{p(|t|-1)}(t')
  .\]

  Let \(\alpha\in\mc D_X^\omega\), and let \(\alpha_i\) be its subsequence
  removing all terms of index \(n\) such that \(p(n)\not=i\).
  By \(sCDFT_{d_i}^+\), we have
  \[
    d_i
      \in
    \cl{\{\sigma_i(\alpha_i\rest (n+1)):n<\omega\}}
      =
    \cl{\{\tau(\alpha\rest(n+1)):n\in p^{\leftarrow}(i)\}}
      \subseteq
    \cl{\{\tau(\alpha\rest(n+1)):n<\omega\}}
  \]
  and as \(D\subseteq\cl{\{\tau(\alpha\rest(n+1)):n<\omega\}}\)
  it follows that
  \[
    X
      \subseteq
    \cl D
      \subseteq
    \cl{\cl{\{\tau(\alpha\rest(n+1)):n<\omega\}}}
      =
    \cl{\{\tau(\alpha\rest(n+1)):n<\omega\}}
  .\]
  Therefore \(\tau\) witnesses \(SS^+\).

  Let \(\sigma_i\) be a witness for \(sCDFT_{d_i}^{+mark}\)
  for each \(i<\omega\). Define \(p:\omega\to\omega\) partition \(\omega\)
  into infinite sets \(\{p(i):i<\omega\}\).
  We then define the Markov strategy \(\tau\) by
  \[
    \tau(D,n)
      =
    \sigma_{p(n)}(D,|\{m\in p(n):m\leq n\}|)
  .\]

  Let \(\alpha\in\mc D_X^\omega\), and let \(\alpha_i\) be its subsequence
  removing all terms of index \(n\) such that \(p(n)\not=i\).
  By \(sCDFT_{d_i}^{+mark}\), we have
  \[
    d_i
      \in
    \cl{\{\sigma_i(\alpha_i(n),n):n<\omega\}}
      =
    \cl{\{\tau(\alpha(n),n):n\in p^{\leftarrow}(i)\}}
      \subseteq
    \cl{\{\tau(\alpha(n),n):n<\omega\}}
  \]
  and as \(D\subseteq\cl{\{\tau(\alpha(n),n):n<\omega\}}\)
  it follows that
  \[
    X
      \subseteq
    \cl D
      \subseteq
    \cl{\cl{\{\tau(\alpha(n),n):n<\omega\}}}
      =
    \cl{\{\tau(\alpha(n),n):n<\omega\}}
  .\]
  Therefore \(\tau\) witnesses \(SS^{+mark}\).
\end{proof}

As mentioned earlier, the following is a result of Sakai:

\begin{theorem}[\cite{MR964873}]
  The following are equivalent for any \(T_{3\frac{1}{2}}\)
  topological space \(X\).
    \begin{itemize}
      \item \(X\) is \(\Omega R\).
      \item \(C_p(X)\) is \(sCFT\).
      \item \(C_p(X)\) is \(sCDFT\).
    \end{itemize}
\end{theorem}

The corresponding game-theoretic results also hold.

\begin{theorem}
  The following are equivalent for any \(T_{3\frac{1}{2}}\)
  topological space \(X\).
    \begin{itemize}
      \item \(X\) is \(\Omega R\)
            (resp. \(\Omega R^+\), \(\Omega R^{+mark}\)).
      \item \(C_p(X)\) is \(sCFT\)
            (resp. \(sCFT^+\), \(sCFT^{+mark}\)).
      \item \(C_p(X)\) is \(sCDFT\)
            (resp. \(sCDFT^+\), \(sCDFT^{+mark}\)).
    \end{itemize}
\end{theorem}

\begin{proof}
  In Lemmas \ref{mengerFanLemma1} and \ref{mengerFanLemma2},
  \(|\tau(t)|=|\sigma(t)|\). Therefore they may be extended to
  the strong cases requiring \(|\tau(t)|=|\sigma(t)|=1\),
  which yields our result.
\end{proof}

\begin{theorem}
  Let \(\omega_1\) have the discrete topology. Then
  \(\onePtLind{\omega_1}\) is \(\Omega R^+\) but not \(\Omega M^{+mark}\).
\end{theorem}

\begin{proof}
  The strategy constructed in Theorem \ref{omega1OmegaMenger} is a witness.
\end{proof}

\begin{corollary}
  \(C_p(\onePtLind{\omega_1})\) is \(sCFT^+\) but not \(CDFT^{+mark}\).
\end{corollary}

\begin{lemma}
  Let \(\mc A\) be a family of sets where
  \(|\bigcup\mc A|\leq\omega\)
  and let \(\mc B\) be a family of sets closed under supersets. Then
  \(\plII\win\schStrongSelGame{\mc A}{\mc B}\) if and only if
  \(\plII\markwin\schStrongSelGame{\mc A}{\mc B}\).
\end{lemma}

\begin{proof}
  In the proof of Lemma \ref{mainLemma}, \(|\tau(t)|=|\sigma(t)|\).
\end{proof}

\begin{corollary}
  A countable space is \(sSS^+\) if and only if it is
  \(sSS^{+mark}\).
\end{corollary}

\begin{corollary}
  A second-countable space is \(\Omega R^+\) if and only if
  it is \(\Omega R^{+mark}\).
\end{corollary}

\begin{corollary}
  A second-countable space is \(R^+\) if and only if
  it is \(R^{+mark}\).
\end{corollary}

\section{Acknowledgements}
The author would like to thank Ziqin Feng and Gary Gruenhage
for their helpful comments on the topic of \(C_p(\onePtLind{\omega_1})\).

\bibliographystyle{plain}
\bibliography{../../bibliography}

\end{document}